\documentclass[11pt]{article}
\usepackage{amsmath,amsthm,amssymb,bbm}

\def\Ex{{\mathbb E}}
\def\Pr{{\mathbb P}}
\def\ind{\mathbbm{1}}
\def\er{{\mathbb R}}
\def\ve{\varepsilon}
\def\cale{{\cal E}}
\def\sgn{\mathrm{sgn}}
\def\supp{\mathrm{supp}}
\def\Cov{\mathrm{Cov}}
\def\Med{\mathrm{Med}}
\def\dim{\mathrm{dim}}
\def\vol{\mathrm{vol}}
\def\de{\mathrm{d}}

\newtheorem{thm}{Theorem}[section]
\newtheorem{lem}[thm]{Lemma}
\newtheorem{prop}[thm]{Proposition}
\newtheorem{conj}[thm]{Conjecture}
\newtheorem{cor}[thm]{Corollary}

\theoremstyle{definition}
\newtheorem{exm}[thm]{Example}
\newtheorem{dfn}[thm]{Definition}

\theoremstyle{remark}
\newtheorem{rem}[thm]{Remark}

\title{Sudakov-type minoration for log-concave vectors
\thanks{Research supported by the NCN grant DEC-2012/05/B/ST1/00412\newline
AMS 2010 Classification: Primary 60E15; Secondary 46B09, 52A23\newline
Keywords: log-concave random vectors, covering numbers, chaining methods}}
\author{Rafa{\l} Lata{\l}a}
\date{}

\begin{document}

\maketitle

\begin{abstract}
We formulate and discuss a conjecture concerning lower bounds for norms of log-concave vectors, which
generalizes the classical Sudakov minoration principle for Gaussian vectors. We show that the conjecture holds for some
special classes of log-concave measures and some weaker forms of it are satisfied in the general case. We also present some
applications based on chaining techniques.  
\end{abstract}

\section{Introduction and formulation of the problem}

In numerous problems arising in high dimensional probability one needs to estimate $\Ex\|X\|$, where $X$ is a random $d$-dimensional vector and
$\|\ \|$ is a norm on $\er^d$. Obviously $\|x\|=\sup_{\|t\|_*\leq 1}\langle t,x\rangle$, so the question reduces to
finding bounds for $\Ex\sup_{t\in T}\langle t,X\rangle$ with $T\subset\er^d$. For symmetric random vectors this quantity is 
a half of $\Ex\sup_{t,s\in T}\langle t-s,X\rangle$, however in the case of arbitary (not necessary centered) random vectors
it is more convienient to work with the latter quantity.

There are numerous powerful methods to estimate suprema of stochastic processes 
(cf. the monograph \cite{TaBook}),
let us however present only a very easy upper bound. Namely for any $p\geq 1$,
\begin{align*}
\Ex\sup_{t,s\in T}\langle t-s,X\rangle&=\Ex\sup_{t,s\in T}|\langle t-s,X\rangle|\leq 
 \Big(\Ex\sup_{t,s\in T}|\langle t-s,X\rangle|^p\Big)^{1/p}
\\
&\leq \Big(\Ex\sum_{t,s\in T}|\langle t-s,X\rangle|^p\Big)^{1/p}
\leq |T|^{2/p}\sup_{t,s\in T}\|\langle t-s,X\rangle\|_p.
\end{align*}
Here and in the sequel $\|Y\|_p:=(\Ex|Y|^p)^{1/p}$ for a real random variable $Y$ and $p>0$. In particular
\[
\Ex\sup_{t,s\in T}\langle t-s,X\rangle\leq e^2 \sup_{t,s\in T}\|\langle t-s,X\rangle\|_p \quad \mbox{ if }|T|\leq e^p.
\]

It is natural to ask when the above estimate may be reversed. 
Namely, when is it true that if the set $T\subset \er^d$
has large cardinality (say at least $e^p$) and variables $(\langle t,X\rangle)_{t\in T}$ are $A$-separated with respect to 
the $L_p$-distance then $\Ex\sup_{t,s\in T}\langle t,X\rangle$ is at least of the order of $A$? The following definition
gives a more precise formulation of such property.

\begin{dfn}
\label{defSM}
Let $X$ be a random $d$-dimensional vector. We say that $X$ 
\emph{satisfies the $L_p$-Sudakov minoration principle with a constant $\kappa>0$} ($\mathrm{SMP}_{p}(\kappa)$ in short) 
if for any set $T\subset \er^d$ with $|T|\geq e^p$ such that
\begin{equation}
\label{sudass}
\|\langle t-s,X\rangle\|_p=\Big(\Ex\Big(\sum_{i=1}^d(t_i-s_i)X_i\Big)^p\Big)^{1/p}\geq A
\quad \mbox{for all } s,t\in T,\ s\neq t,
\end{equation}
we have
\begin{equation}
\label{sudconcl}
\Ex \sup_{t,s\in T}\langle t-s,X\rangle=\Ex \sup_{t,s\in T}\sum_{i=1}^d (t_i-s_i)X_i
\geq \kappa A.
\end{equation}
A random vector $X$ \emph{satisfies the Sudakov minoration principle with a constant $\kappa$} ($\mathrm{SMP}(\kappa)$ in short)
if it satisfies $\mathrm{SMP}_{p}(\kappa)$ for any $p\geq 1$.
\end{dfn}

\begin{rem}
One cannot hope to improve the estimate \eqref{sudconcl} even if $X$ has a regular product distribution 
and $|T|$ is very
large with respect to $p$. To see this take $X$ uniformly distributed on the cube $[-1,1]^d$,
then for $p\geq 1$, $\|X_i-X_j\|_{p}\geq \|X_i-X_j\|_1=2/3$ for all $1\leq i<j\leq d$ and
$\Ex\sup_{i,j\leq n}(X_i-X_j)\leq 2$.
\end{rem}

\begin{exm} 
If $X$ has the canonical  $d$-dimensional Gaussian distribution then 
$\|\langle t,X\rangle\|_p=\gamma_p|t|$, where $\gamma_p=\|{\cal N}(0,1)\|_p\sim \sqrt{p}$ for $p\geq 1$.
Hence condition \eqref{sudass} is equivalent to $|t-s|\geq A/\gamma_p$ for distinct vectors $t,s\in T$ and the
classical Sudakov minoration principle for Gaussian processes, cf. \cite{Su} and \cite[Theorem 3.18]{LT},
then yields
\[
\Ex  \sup_{t,s\in T}\langle t-s,X\rangle=
2\Ex \sup_{t\in T}\langle t,X\rangle\geq \frac{A}{C\gamma_p}\sqrt{\log |T|}\geq \frac{A}{C'}
\]
provided that $|T|\geq e^{p}$ ($C$ and $C'$ denote universal constants). Therefore $X$ satisfies  the
Sudakov minoration principle with a universal constant. In fact it is not hard to see that
for centered Gaussian vectors the Sudakov minoration principle in the sense of Definition \ref{defSM} is
formally equivalent to the minoration property established by Sudakov.  
\end{exm}

\begin{exm}
\label{smprad}
If $X_i$'s are independent symmetric $\pm 1$ r.v.'s (equivalently one may
consider the vector $X$ uniformly distributed on the cube $[-1,1]^d$) then
condition \eqref{sudass} means, by the result of Hitczenko \cite{H}, that 
$t-s\notin \frac{A}{C}(B_1^n+\sqrt{p}B_2^n)$ and
in this case $\mathrm{SMP}(\kappa)$ with universal $\kappa$ was proven by Talagrand  \cite{Ta1}. 
\end{exm}

\begin{exm}
\label{indlogcon}
In the more general case when coordinates of $X$ are independent and symmetric 
with log-concave densities (or just log-concave tails) the Sudakov minoration priciple with
a universal constant was proven in 
\cite{Ta2} (for random variables with the density $\exp(-|x|^p)$, $p\geq 1$) and \cite{La}. 
\end{exm}

The Sudakov minoration principle for vectors $X$ with independent coordinates  is investigated in \cite{LaT},
where it is shown that $\mathrm{SMP}$ is essentially equivalent to the regular
growth of moments of coordinates of $X$. In this paper we will concentrate on the class of log-concave vectors.  

\medskip

A measure $\mu$ on $\er^n$ is called \emph{logarithmically concave} (or \emph{log-concave} in short) if
$\mu(\lambda K+(1-\lambda)L)\geq \mu(K)^{\lambda}\mu(L)^{1-\lambda}$ for any nonempty compact
sets $K,L$ and $\lambda\in [0,1]$. By the result of Borell \cite{Bo} a measure on $\er^n$ with the full-dimensional
support is log-concave if and only if it has a log-concave density, i.e. the density of the form
$e^{-h(x)}$, where $h\colon \er^n\rightarrow (-\infty,\infty]$ is convex. 
A random vector is called log-concave if its distribution is logarithmically concave. A typical example of a 
log-concave vector is a vector uniformly distributed on a convex body.

It is quite easy to reduce investigation of the Sudakov minoration principle to the case of symmetric vectors, 
see Lemma \ref{red_to_symm} below. 
Since SMP is preserved under linear transformations (Lemma \ref{preslin}) 
we may additionally assume that the vector $X$ is \emph{isotropic}, i.e. $\Cov(X_i,X_j)=\delta_{i,j}$ for all $i,j$.
In many aspects isotropic log-concave probability measures behave like product measures, cf. \cite{BGVV}. This motivates
the following conjecture.

\begin{conj}
\label{Sudmin}
Every $d$-dimensional log-concave random vector satisfies the Sudakov-minoration principle with 
a universal constant.
\end{conj}

The purpose of this paper is to discuss the above conjecture. 
In Section 2 we gather simple facts concerning log-concave vectors and the Sudakov minoration principle. In particular we
show how to reduce the problem to the case of isotropic vectors. In Section 3 we establish several results concerning 
arbitrary log-concave distribution. We show that \eqref{sudass} implies \eqref{sudconcl} provided that $|T|\geq e^{e^p}$
or $|T|\geq e^p$, but under the additional assumption that vectors $(\langle t,X\rangle)_{t\in T}$ are uncorrelated. 
The proof is based on the concentration
properties of isotropic log-concave distributions. As a byproduct we get a comparison of weak and strong moments
of $\ell_{\infty}^d$-norms of isotropic log-concave vectors. In Section 4 we consider unconditional log-concave vectors.
We show that in this case \eqref{sudass} implies \eqref{sudconcl} provided that $|T|\geq e^{p^2}$. In Section 5
we show that Conjecture \ref{Sudmin} holds for a class of invariant log-concave vectors, which includes rotationally
invariant log-concave vectors and vectors uniformly distributed on $l_p^d$-balls. In the last section we use chaining 
arguments to show some consequences of the Sudakov minoration principle. In particular we show that it yields comparison
of weak and strong moments up to a logarithmic factor.

It should be mentioned that the Sudakov minoration principle and Conjecture \ref{Sudmin} was formulated independently
and studied by Shahar Mendelson, Emanuel Milman and Grigoris Paouris \cite{MMP}. Their approach is however quite different,
uses geometrical properties of an index set $T$, duality of entropy numbers and the idea of dimension reduction, similar
in the spirit to the Johnson-Lindenstrauss lemma.    

\medskip

\noindent
{\bf Notation.}  By $|\cdot|$ and $\langle \cdot , \cdot \rangle$ we denote the canonical
Euclidean norm and the canonical inner product on $\er ^d$.  The canonical basis of $\er ^d$ is denoted 
by $e_1, \ldots, e_d$. For $1\leq p\leq \infty$, $\|\cdot\|_p$ stands for the $l_p$ norm on $\er^d$ and $B_p^d$
is the unit ball in this norm. 

For two convex sets $K,L$ in $\er^d$, $N(K,L)$ is the covering number, i.e.\ the minimal number of translates of $L$ that 
cover $K$. By $|T|$ we denote the cardinality of a set $T$ and by $N(T,d,\ve)$ the minimal number of balls in metric $d$ 
of radius $\ve$ that cover $T$. 

We use letter $C$ for universal constants, value of a constant $C$ may differ at each
occurence. Whenever we want to fix the value of an absolute constant we  use letters 
$C_1,C_2,\ldots$.

\section{Basic Facts} 

We start with a lemma showing how to reduce the problem of proving the Sudakov minoration to the case of symmetric vectors.

\begin{lem}
\label{red_to_symm}
Let $p\geq 1$, $X$ be a random vector in $\er^d$ with finite $p$-th moment and $X'$ be an independent copy of $X$.
If $X-X'$ satisfies $\mathrm{SMP}_p(\kappa)$ then $X$ satisfies $\mathrm{SMP}_p(\min\{1/2,\kappa/4\})$.
\end{lem}

\begin{proof}
Let $p\geq 1$, $T\subset \er^d$ be such that $|T|\geq e^p$ and \eqref{sudass} holds. Jensen's inequality yields 
\[
\Ex \sup_{t,s\in T}\langle t-s,X\rangle\geq \sup_{t,s\in T}\langle t-s, \Ex X\rangle=
\sup_{t,s\in T}|\langle t-s,\Ex X\rangle|.
\]
Therefore we may assume that $|\langle t-s, \Ex X\rangle|\leq A/2$ for all $t,s\in T$. But then for $t\neq s$, $t,s\in T$,
\[
\|\langle t-s,X-X'\rangle\|_p\geq \|\langle t-s, X-\Ex X\rangle\|_p\geq 
\|\langle t-s,X\rangle\|_p-\|\langle t-s,\Ex X\rangle\|_p\geq \frac{A}{2}.
\]
Therefore the $L_p$-Sudakov minoration for $X-X'$ implies
\begin{align*}
\kappa\frac{A}{2}&\leq \Ex\sup_{t,s\in T}\langle t-s,X-X'\rangle\leq \Ex\sup_{t,s\in T}\langle t-s,X\rangle+
\Ex\sup_{t,s\in T}\langle s-t,X'\rangle
\\
&=2\Ex\sup_{t,s\in T}\langle t-s,X\rangle.
\end{align*}
\end{proof}

The next observation states that the Sudakov minoration principle is preserved under linear transformations.

\begin{lem}
\label{preslin}
If $X$ is a $d$-dimensional random vector that satisfies $\mathrm{SMP}_p(\kappa)$ and $U\colon \er^d\rightarrow \er^{d'}$
is linear then $UX$ satisfies $\mathrm{SMP}_p(\kappa)$.
\end{lem}

\begin{proof} It is enough to observe that $\langle t,UX\rangle=\langle U^*t,X\rangle$.
\end{proof}

Now we recall a fact that moments of log-concave variables growth in a regular way.

\begin{lem}
\label{momgrow}
Let $Y$ be a symmetric real log-concave r.v. Then 
\[
\|Y\|_p\leq \frac{\Gamma(p+1)^{1/p}}{\Gamma(q+1)^{1/q}}\|Y\|_q\quad \mbox{for }p\geq q>0.
\]
In particular $\|Y\|_2\leq \sqrt{2}\|Y\|_1$ and $\|Y\|_p\leq \frac{p}{q}\|Y\|_q$ for $p\geq q\geq 2$.
\end{lem}

\begin{proof}
The main inequality is the result of Barlow, Marshall and Proschan \cite{BMP}
(it may be also extracted from the much earlier work of Berwald \cite{Ber}). 
To show the ``in particular"
part for $p\geq q\geq 2$ one needs to estimate $\Gamma$ functions as it was done in \cite[Proposition 3.8]{LaW}.
\end{proof}

\begin{rem}
\label{smallp}
Suppose that $|T|\geq 2$ and \eqref{sudass} holds. Then if $X$ is symmetric log-concave
we may choose $t^1,t^2\in T$ with $t^1\neq t^2$ and get by Lemma \ref{momgrow}
\[
\Ex \sup_{t,s\in T}\langle t-s,X\rangle
\geq \Ex |\langle t^2-t^1,X\rangle|
\geq \frac{\sqrt{2}}{\max\{p,2\}}\|\langle t^2-t^1,X\rangle\|_p
\geq \frac{\sqrt{2}}{\max\{p,2\}}A.
\]
Hence every symmetric log-concave vector satisfies $\mathrm{SMP}_p(\sqrt{2}/\max\{p,2\})$ and every
log-concave vector satisfies $\mathrm{SMP}_p(\sqrt{2}/\max\{4p,8\})$.
\end{rem}

\begin{rem}
\label{equivcov}
Let us define for $p\geq 1$ the distance on $\er^d$
\[
d_{X,p}(s,t):=\|\langle s-t,X\rangle\|_p.
\]
Suppose that \eqref{sudass} is satisfied, but $ |T|= e^q$ with $1\leq q\leq p$. 
We know that $d_{X,q}(s,t)\geq \frac{q}{Cp}d_{X,p}(s,t)$ so the Sudakov minoration principle
for a log-concave vector $X$  implies
the  following formally stronger statement -- for any nonempty $T\subset \er^d$ and $A>0$,
\[
\Ex\sup_{t,s\in T}\langle t-s,X\rangle \geq \frac{\kappa}{C}\sup_{p\geq 1}\min\Big\{\frac{A}{p}\log N(T,d_{X,p},A),A\Big\}.
\]
\end{rem}

The next result says that it is enough to verify the Sudakov minoration property only for $p\leq d$.

\begin{lem}
\label{red_belowd}
Let $X$ be a symmetric log-concave random vector in $\er^d$ that satisfies $\mathrm{SMP}_d(\kappa)$.
Then $X$ satisfies $\mathrm{SMP}_p(\kappa/8)$ for $p\geq d$.  
\end{lem}

\begin{proof}
Fix $p\geq d$ and $T\subset \er^d$ such that $|T|\geq e^p$ and $\|\langle s-t,X\rangle\|_p\geq A$ for any distinct points $s,t\in T$.  
Let 
\[
{\cal M}_p(X):=\{t\in \er^d\colon\ \Ex|\langle t,X\rangle|^p\leq 1\}.
\]

If $d\leq p\leq 8d$ then by Lemma \ref{momgrow}, $\|\langle s-t,X\rangle\|_d\geq \frac{1}{8}A$, hence $\mathrm{SMP}_d(\kappa)$ 
yields $\Ex\sup_{t,s\in T}\langle t-s,X\rangle \geq \frac{\kappa}{8}A$.

If $p\geq 8d$ then we have for $u\geq 1$,
\begin{align*}
e^p&\leq N\Big(T,\frac{A}{2}{\cal M}_p(X)\Big)\leq 
N\Big(T,u\frac{A}{2}{\cal M}_p(X)\Big)N\Big(u\frac{A}{2}{\cal M}_p(X),\frac{A}{2}{\cal M}_p(X)\Big)
\\
&\leq N\Big(T,u\frac{A}{2}{\cal M}_p(X)\Big)(2u+1)^d,
\end{align*}
where the last inequality follows by the standard volumetric argument.
This shows that $N(T,u\frac{A}{2}{\cal M}_p(X))\geq e^d$ if $u\leq e^{p/(4d)}$, therefore we may find $T_1\subset T$
with $|T_1|\geq e^d$ such that for all $s,t\in T_1$, $s\neq t$,
\[
\|\langle s-t,X\rangle\|_d\geq  \frac{d}{p}\|\langle s-t,X\rangle\|_p \geq \frac{d}{p}e^{p/(4d)}\frac{A}{2}\
\geq \frac{A}{8}.
\]
Thus again $\mathrm{SMP}_d(\kappa)$ yields
\[
\Ex\sup_{t,s\in T}\langle t-s,X\rangle \geq \Ex\sup_{t,s\in T_1}\langle t-s,X\rangle \geq \kappa\frac{A}{8}.
\]
\end{proof}

\begin{rem}
Lemmas \ref{red_to_symm} and \ref{red_belowd} together with Remark \ref{smallp} show that every log-concave vector
satisfy $\mathrm{SMP}(1/(Cd))$. 
\end{rem}

The following easy observation shows that Sudakov minoration holds with a universal constant if $p$ is large with respect
to the dimension $d$.

\begin{lem}
Every symmetric $d$-dimensional log-concave vector $X$ satisfies for $p\geq 2$, 
$\mathrm{SMP}_p(\frac{1}{\sqrt{2}p}(e^{p/d}-1))$. In particular $X$ satisfies 
$\mathrm{SMP}_p(\frac{1}{3})$ for $p\geq 2d\log(d+e)$.
\end{lem}

\begin{proof}
By Lemma \ref{preslin} we may assume that $X$ is isotropic. Assume that $|T|\geq e^p$, $p\geq 2$ and
\eqref{sudass} holds. Then by Lemma \ref{momgrow}
\[
|t-s|=\|\langle t-s,X\rangle\|_2\geq \frac{2}{p}\|\langle t-s,X\rangle\|_p
\geq \frac{2}{p}A.
\]

This shows that the sets $(t+\frac{A}{p}B_2^d)_{t\in T}$ have disjoint interiors. Standard volumetric argument
gives that there exist $t,s\in T$ such that $|t-s|\geq \frac{A}{p}(|T|^{1/d}-1)$.
We have
\[
\Ex\sup_{t,s\in T}\langle t-s,X\rangle= \sup_{t,s\in T}\Ex|\langle t-s,X\rangle|\geq
\frac{1}{\sqrt{2}}\sup_{t,s\in T}|t-s|\geq \frac{A}{\sqrt{2}p}(e^{p/d}-1).
\]

Finally if $p=ad\log(d+e)$ with $a\geq 2$ then $p\geq 2$ and 
$e^{p/d}-1\geq d(d+e)^{a-1}\geq d(a-1)\log(d+e)\geq \frac{1}{2}p$.
\end{proof}

\section{Estimates for general log-concave measures}

We say that a random vector $X$ in $\er^d$ satisfies exponential concentration with a constant 
$\alpha<\infty$ if for any Borel set $B$ in $\er^d$,
\[
\Pr(X\in B)\geq \frac{1}{2}\ \Rightarrow\ \Pr(X\in B+\alpha uB_2^d)\geq 1-e^{-u}\  
\mbox{ for }  u>0. 
\]
It is an important open problem \cite{KLS} whether isotropic log-concave vectors satisfy exponential concentration 
with a universal constant. E.~Milman \cite{Mi} showed that this problem has numerous equivalent functional
and isoperimetrical formulations. Klartag \cite{Kl} proved that every isotropic $d$-dimensional log-concave vector satisfies
exponential concentration with a constant $\alpha\leq Cd^{1/2-\ve}$ with $\ve\geq 1/30$. This bound was improved
by Eldan \cite{El} to $\alpha\leq Cd^{1/3}\log^{1/2}(d+1)$.

We start this section with deriving a simple consequence of exponential concentration, which will be used in the
sequel to estimate $\ell_\infty$-norms of log-concave vectors.

\begin{prop}
\label{conc}
Suppose that a random vector $X$ satisfies exponential concentration with a constant 
$\alpha$. Then for any $V>0$, $p\geq 2$ and $T\subset\er^d$ we have
\begin{align*}
\Big(\Ex\sum_{t\in T}(|\langle t, X\rangle|\wedge V)^p\Big)^{1/p}
&\leq 
 2\Ex\Big(\sum_{t\in T}(|\langle t, X\rangle|\wedge V)^p\Big)^{1/p}
\\
&\phantom{aa} +2^{1/p}V^{(p-2)/p}(\alpha\beta(T))^{2/p},
\end{align*}
where
\[
\beta(T):=\sup_{|x|=1}\Big(\sum_{t\in T}|\langle t, x\rangle|^2\Big)^{1/2}.
\]
\end{prop}

\begin{proof}
Let 
\[
S:= \Big(\sum_{t\in T}(|\langle t, X\rangle|\wedge V)^p\Big)^{1/p} \quad
\mbox{ and }\quad
M:=\Ex S.
\]
Define also
\[
B:=\Big\{x\in \er^n\colon\ \Big(\sum_{t\in T}(|\langle t, x\rangle|\wedge V)^p\Big)^{1/p}\leq 2M \Big\}.
\]
Then $\Pr(X\in B)\geq 1/2$. Notice that for $x=y+z\in B+ uB_2^n$, $u>0$ we have
\begin{align*}
\Big(\sum_{t\in T}(|\langle t, x\rangle|\wedge V)^p\Big)^{1/p}&\leq
\Big(\sum_{t\in T}(|\langle t, y\rangle|\wedge V)^p\Big)^{1/p}
+\Big(\sum_{t\in T}(|\langle t, z\rangle|\wedge V)^p\Big)^{1/p}
\\
&\leq
2M+V^{(p-2)/p}\Big(\sum_{t\in T}|\langle t, z\rangle|^2\Big)^{1/p}
\\
&\leq 2M+V^{(p-2)/p}(u\beta(T))^{2/p}.
\end{align*}
Hence exponential concentration yields
\[
\Pr\Big(S\geq 2M+V^{(p-2)/p}(\alpha\beta(T)u)^{2/p}\Big)\leq e^{-u}
\quad\mbox{ for }u>0.
\]
Integrating by parts this gives
\begin{align*}
&(\Ex(S-2M)_{+}^p)^{1/p}
\\
&\leq 
V^{(p-2)/p}(\alpha\beta(T))^{2/p}\Big(p\int_{0}^\infty u^{p-1}\Pr(S\geq 2M+V^{(p-2)/p}(\alpha\beta(T))^{2/p}u)\de u\Big)^{1/p}
\\
&\leq V^{(p-2)/p}(\alpha\beta(T))^{2/p}
\Big(p\int_{0}^\infty u^{p-1}e^{-u^{p/2}}\de u\Big)^{1/p}=2^{1/p}V^{(p-2)/p}(\alpha\beta(T))^{2/p}.
\end{align*}
\end{proof}

We also need a simple technical lemma.

\begin{lem}
\label{lpwedge}
Suppose that $Y$ is a real symmetric log-concave r.v., $p\geq 2$ and $\|Y\|_p\geq V>0$.
Then $\Ex(|Y|\wedge V)^p\geq (V/12)^{p}.$
\end{lem}

\begin{proof}
By Lemma \ref{momgrow}, $\|Y\|_{2p}\leq 2\|Y\|_p$, hence  the Paley-Zygmund inequality yields
\[
\Pr(|Y|\geq 2^{-1/p}V)\geq \Pr\Big(|Y|^p\geq \frac{1}{2}\Ex|Y|^p\Big)\geq \frac{(\Ex|Y|^p)^2}{4\Ex|Y|^{2p}}\geq
\frac{1}{4\cdot 2^{2p}}
\]
and 
\[
\Ex(|Y|\wedge V)^p\geq \frac{1}{2}V^p\Pr(|Y|\geq 2^{-1/p}V)\geq \frac{1}{8\cdot 4^p}V^p\geq 
\Big(\frac{V}{12}\Big)^p.
\]
\end{proof}

We are now ready to state a lower bound for suprema of coordinates of isotropic log-concave vectors.

\begin{prop}
\label{supcoord}
Let $X$ be an isotropic log-concave random vector in $\er^d$. Suppose that 
$p\geq 2$, $d\geq e^p-1$ and $\|a_iX_i\|_p\geq V$ for $1\leq i\leq d$. Then 
\[
\Ex\max\{|a_1X_1|,\ldots,|a_dX_d|\}\geq\frac{1}{C_1}V.
\]
\end{prop}

\begin{proof}
Symmetrization argument as in Lemma \ref{red_to_symm} shows that we may additionally assume
that $X$ is symmetric. By Lemma \ref{momgrow}
\[
\Ex\max_{i}|a_iX_i|\geq \max_{i}\|a_iX_i\|_1\geq \frac{\sqrt{2}}{p}\max_{i}\|a_iX_i\|_p
\geq \frac{\sqrt{2}}{p}V,
\]
so we may assume that $p$ (and therefore also $d$) is sufficiently large. 
Since $e^{p}-1\geq e^{p/2}$ and by Lemma \ref{momgrow}, 
$\|a_iX_i\|_{\lambda p}\geq \lambda\|a_iX_i\|_p\geq \lambda V$ for $\lambda\in (0,1]$ and $p\geq 2/\lambda$,  
we may assume (changing $p$ and to $p/8\ln(24)$  and $V$ to $V/8\ln(24)$) that $d\geq 24^{4p}$. 

Since it is only a matter of normalization of coefficients $a_i$ and the number $V$ we may and will assume that 
$\max_i|a_i|=1$. But then by Lemma \ref{momgrow} 
\[
\Ex\max\{|a_1X_1|,\ldots,|a_dX_d|\}\geq \max_{i}\Ex|a_iX_i|\geq \min_i\Ex|X_i|\geq \frac{1}{\sqrt{2}},
\]
so it is enough to consider the case $V\geq 2$.

By the result of Eldan \cite{El} $X$ satisfies exponential concentration with constant at most
$d^{1/2-1/8}$ (recall that we assume that $d$ is sufficiently large). Let 
$T:=\{a_ie_i\colon i\leq d\}\subset \er^d$ then
\[
\beta(T)^2=\sup_{|x|=1}\sum_{i=1}^d|a_ix_i|^2=\max_{i}a_i^2=1,
\] 
hence Proposition \ref{conc} yields (recall that $V\geq 2$)
\[
\Big(\Ex\sum_{i=1}^d(|a_iX_i|\wedge V)^p\Big)^{1/p}
\leq 2\Ex\Big(\sum_{i=1}^d(|a_iX_i|\wedge V)^p\Big)^{1/p}+Vd^{1/p-1/(4p)}.
\]
We have 
\[
\Ex\Big(\sum_{i=1}^d(|a_iX_i|\wedge V)^p\Big)^{1/p}\leq d^{1/p}\Ex\max_{1\leq i\leq d}|a_iX_i|.
\]
By Lemma \ref{lpwedge} we know that $\Ex(|a_iX_i|\wedge V)^p\geq (V/12)^p$, therefore
\[
\Big(\Ex\sum_{i=1}^d(|a_iX_i|\wedge V)^p\Big)^{1/p}\geq \frac{1}{12}Vd^{1/p}.
\]
Thus
\[
\frac{1}{12}Vd^{1/p}\leq 2 d^{1/p}\Ex\max_{1\leq i\leq d}|a_iX_i|+Vd^{1/p-1/(4p)}.
\]
However $d^{1/(4p)}\geq 24$ and we get
\[
\Ex\max_{1\leq i\leq d}|a_iX_i|\geq \frac{1}{48}V.
\]
\end{proof}

As a corollary we show that Conjecture \ref{Sudmin} holds for sets $T$ such that r.v's $(\langle t,X\rangle)_{t\in T}$
are uncorellated.

\begin{cor}
\label{sudminuncorr}
Suppose that $X$ is a $d$-dimensional log-concave random vector, $p\geq 2$, $T\subset \er^d$
satisfies \eqref{sudass} and $\Cov(\langle t,X\rangle,\langle s,X\rangle)=0$ for $s,t\in T$ with $s\neq t$.
Then \eqref{sudconcl} holds with a universal constant $\kappa$ provided that $|T|\geq e^p$.
\end{cor}

\begin{proof}
Using symmetrization argument as in the proof of Lemma \ref{red_to_symm} we may assume that
$X$ is symmetric.
 
Since $\|\langle t-s,X\rangle\|_p\leq \|\langle t,X\rangle\|_p+\|\langle s,X\rangle\|_p$, 
there exist $t_1,\ldots,t_n\subset T$ with $n\geq |T|-1\geq e^p-1$ such that 
$\|\langle t_i,X\rangle\|_p\geq A/2$ for all $i$.
Proposition \ref{supcoord} applied with $V=A/2$, $a_i:=\|\langle t_i,X\rangle\|_2$ and $n$-dimensional
isotropic vector $Y=(\langle t_i,X\rangle/a_i)_{i\leq n}$ gives
\[
\Ex\max_{t\in T}|\langle t,X\rangle|\geq \Ex\max_{i}|\langle t_i,X\rangle|=\Ex\max_i|a_iY_i|\geq \frac{1}{2C_1}A.
\] 
Notice that for any $t_0\in T$ we have
\[
\Ex\max_{t,s\in T}\langle t-s,X\rangle\geq \Ex\max_{t\in T}|\langle t-t_0,X\rangle|\geq 
\Ex\max_{t\in T}|\langle t,X\rangle|-\Ex|\langle t_0,X\rangle|.
\]
If $\Ex|\langle t_0,X\rangle|\leq A/(4C_1)$ we are done, otherwise we may assume that $T\ni t_1\neq t_0$ and get
by Lemma \ref{momgrow}
\begin{align*}
\Ex\max_{t,s\in T}\langle t-s,X\rangle&\geq \Ex|\langle t_1-t_0,X\rangle|
\geq \frac{1}{\sqrt{2}}\|\langle t_1-t_0,X\rangle\|_2\geq \frac{1}{\sqrt{2}}\|\langle t_0,X\rangle\|_2
\\
&\geq \frac{1}{\sqrt{2}}\Ex|\langle t_0,X\rangle| \geq \frac{1}{4\sqrt{2}C_1}A,
\end{align*}
where the third inequality follows since $\Cov(\langle t_0,X\rangle,\langle t_1,X\rangle)=0$.
\end{proof}

Before we formulate next consequence of Proposition \ref{supcoord} we show a simple decomposition lemma.

\begin{lem}
\label{ortdec}
Let $r>0$, $\ve\in [0,1)$ and $T\subset rB_2^d$ satisfy $|T|\geq (\frac{2}{\ve}+1)^n$.
Then we can find vectors $t_k,s_k\in T$, $v_k,u_k\in \er^d$, $k=1,\ldots,n$ such that  
$0\neq t_{k}-s_{k}=u_k+v_k$, $|v_k|\leq \ve r$ for all $k$ and vectors $u_1,\ldots,u_n$ are orthogonal. 
\end{lem}

\begin{proof}
We proceed by an induction. We choose for $t_1,s_1$ any two distinct vectors in $T$ and set $u_1:=t_1-s_1$ and 
$v_1:=0$. Suppose that $1\leq l\leq n-1$ and vectors $t_k,s_k,u_k,v_k$ are chosen
for $1\leq k\leq l$. Define $E:=\mathrm{Lin}(u_1,\ldots,u_l)$ then $\dim E\leq l\leq n-1$.
Since in the ball in $E$ of radius $r$ there are at most $(\frac{2}{\ve}+1)^{\dim E}<|T|$ points with mutual distances
at least $\ve r$ there exist distinct vectors $t_{l+1},s_{l+1}$ in $T$ such that $|P_E(t_{l+1}-s_{l+1})|\leq \ve r$, 
where $P_E$ denotes the orthogonal projection onto $E$. We put $v_{l+1}:=P_E(t_{l+1}-s_{l+1})$ and 
$u_{l+1}:=t_{l+1}-s_{l+1}-v_{l+1}$.
\end{proof}

Next theorem is a weaker form of Conjecture \ref{Sudmin}.

\begin{thm}
Let $X$ be a log-concave vector, $p\geq 1$ and $T\subset \er^d$ be such that $|T|\geq e^{e^p}$ and 
$\|\langle t-s,X\rangle\|_p\geq A$ for all distinct $t,s\in T$. Then 
\[
\Ex\sup_{t,s\in T}\langle t-s,X\rangle \geq \frac{1}{C}A.
\] 
\end{thm}

\begin{proof}
Arguments as in the proofs of Lemmas \ref{red_to_symm} and \ref{preslin} show that it is enough to consider
only symmetric and isotropic vectors $X$.
Since the statement is translation invariant w.l.o.g. $0\in T$.

Let $C_1$ be as in Proposition \ref{supcoord}, we may obviously assume that $C_1\geq 1$.
By Remark \ref{smallp} it is enough to consider the case $p\geq 16eC_1$. Put
$p':=p/(8eC_1)$ and $A':=A/(8eC_1)$, Lemma \ref{momgrow} yields
$\|\langle t-s,X\rangle\|_{p'}\geq A'$ for any distinct vectors $s,t\in T$. Moreover
\[
(1+4eC_1p')^{e^{p'}}\leq (e^{p/2})^{e^{p/2}}\leq e^{e^p}\leq |T|.
\]

Since $0\in T$ we have (using again Lemma \ref{momgrow})
\[
\Ex\sup_{t,s\in T}\langle t-s,X\rangle \geq \sup_{t\in T}\Ex|\langle t,X\rangle|\geq
\frac{1}{\sqrt{2}}\sup_{t\in T}\|\langle t,X\rangle\|_2=\frac{1}{\sqrt{2}}\sup_{t\in T}|t|.
\]
Thus we may assume that $T\subset A'B_2^d$. Let $n:=e^{p'}$ and $\ve:=\frac{1}{2eC_1p'}$, then
$|T|\geq (\frac{2}{\ve}+1)^n$. Therefore we may apply Lemma \ref{ortdec} to the set $T$  with $r=A'$
and $\ve$ , $n$ as above and get points $t_k,s_k,u_k$ and $v_k$. We have
\begin{align*}
\Ex\sup_{t,s\in T}\langle t-s,X\rangle&\geq
\Ex\max_{k\leq n}|\langle t_k-s_k,X\rangle|
=\Ex\max_{k\leq n}|\langle u_k+v_k,X\rangle|
\\
&\geq \Ex\max_{k\leq n}|\langle u_k,X\rangle|-\Ex\max_{k\leq n}|\langle v_k,X\rangle|.
\end{align*}

Notice that 
\begin{align*}
\Ex\max_{k\leq n}|\langle v_k,X\rangle|&\leq \Big(\Ex\sum_{k\leq n}|\langle v_k,X\rangle|^{p'}\Big)^{1/p'}
\leq n^{1/p'}\max_{k}\|\langle v_k,X\rangle\|_{p'}
\\
&\leq e\frac{p'}{2}\max_{k}\|\langle v_k,X\rangle\|_{2}=\frac{ep'}{2}\max_{k\leq n}|v_k|
\leq \frac{ep'}{2}\ve A'= \frac{1}{4C_1}A',
\end{align*}
where the third inequality follows by Lemma \ref{momgrow}.
Moreover for any $k$,
\[
\|\langle u_k,X\rangle\|_{p'}\geq \|\langle t_k-s_k,X\rangle\|_{p'}-\|\langle v_k,X\rangle\|_{p'}\geq A'-\frac{p'}{2}\ve A'
\geq  \frac{A'}{2}.
\]
Since $u_k$ are orthogonal and $X$ is isotropic Proposition \ref{supcoord} (applied with $p=p'$, $V=A'/2$, 
$a_k:=|u_k|$ and $n$-dimensional isotropic vector $Y=(\langle u_k,X\rangle/a_k)_{k\leq n}$) yields
\[
\Ex\sup_{k\leq n}|\langle u_k,X\rangle|\geq \frac{1}{2C_1}A'.
\]
Hence 
\[
\Ex\sup_{t,s\in T}\langle t-s,X\rangle \geq \frac{1}{2C_1}A'-\frac{1}{4C_1}A'=\frac{1}{4C_1}A'.
\]
\end{proof}

\begin{rem}
As in Remark \ref{equivcov} we may reformulate the above result in terms of covering numbers -- for any log-concave
vector $X$, any nonempty $T\subset \er^d$ and $A>0$,
\[
\Ex\sup_{t,s\in T}\langle t-s,X\rangle \geq 
\frac{1}{C}\sup_{p\geq 1}\min\Big\{\frac{A}{p}\log(1+\log N(T,d_{X,p},A)),A\Big\}.
\]
\end{rem}

It is an open problem, cf. \cite{La2}, whether weak and strong moments of log-concave vectors are comparable, 
i.e. whether for $p\geq 1$, log-concave $d$-dimensional vectors $X$ and any norm $\|\ \|$ on $\er^d$,
\[
(\Ex\|X\|^p)^{1/p}\leq C\Big(\Ex\|X\|+\sup_{\|t\|_*\leq 1}\big(\Ex|\langle t,X\rangle|^p\big)^{1/p}\Big).
\]
The next result shows that this is the case for weighted $l_{\infty}^d$-norms of isotropic vectors.

\begin{cor}
Let $X$ be an isotropic $d$-dimensional random vector. Then for any numbers $a_1,\ldots,a_d$ and any $p\geq 1$,
\[
\Big(\Ex\max_{i\leq d}|a_iX_i|^p\Big)^{1/p}\leq 
C\Big(\Ex\max_{i\leq d}|a_iX_i|+\max_{i\leq d}\Big(\Ex|a_iX_i|^p\Big)^{1/p}\Big).
\]
\end{cor}

\begin{proof}
Let $M:=\Ex\max_{i\leq d}|a_iX_i|$. Define 
\[
I_q:=\{i\leq d\colon\, \|a_iX_i\|_q> C_1M\}, \quad q\geq 1
\]
and  
\[
I_{\infty}:=\bigcup_{q\geq 1} I_q=\{i\leq d\colon\, \|a_iX_i\|_{\infty}>C_1M\}. 
\]
Then by Proposition \ref{supcoord},  $|I_q|\leq e^q-1$ for $2\leq q<\infty$.

We have
\begin{align*}
\Big(\Ex\max_{i\in I_{2p}}|a_iX_i|^p\Big)^{1/p}
&\leq \Big(\Ex\sum_{i\in I_{2p}}|a_iX_i|^p\Big)^{1/p}\leq
|I_{2p}|^{1/p}\max_{i\in I_{2p}}\|a_iX_i\|_p
\\
&\leq e^2\max_{i\leq d}\|a_iX_i\|_p.  
\end{align*}

Chebyshev's inequality implies for $u>0$,
\[
\Pr(|a_iX_i|\geq u)\leq u^{-q}\|a_iX_i\|_q^q\leq (C_1M/u)^q \quad \mbox{ for }i\notin I_q.
\]
Therefore for $u\geq 2$,
\begin{align*}
\Pr(\max_{i\notin I_{2p}}|a_iX_i|\geq ue^2C_1M)
&\leq \sum_{i\in I_{\infty}\setminus I_{2p}} \Pr(|a_iX_i|\geq ue^2C_1M)
\\
&\leq \sum_{k=1}^{\infty}\sum_{i\in I_{2^{k+1}p}\setminus I_{2^kp}}  \Pr(|a_iX_i|\geq ue^2C_1M)
\\
&\leq |I_{2^{k+1}p}|\Big(\frac{C_1M}{ue^2C_1M}\Big)^{2^kp}\leq \sum_{k=1}^{\infty} u^{-2^kp}\leq 2u^{-2p}.
\end{align*}
Integration by parts yields 
\[
\Big(\Ex\max_{i\notin I_{2p}}|a_iX_i|^p\Big)^{1/p}\leq CM=C\Ex\max_{i\leq d}|a_iX_i|.
\]
\end{proof}

\section{Unconditional case}

In this section we study the Sudakov minoration principle for unconditional log-concave vectors $X$. 
A random $d$-dimensional vector
$X=(X_1,\ldots,X_d)$ is called \emph{unconditional} if a vector
$(\eta_1 X_1,\ldots,\eta_d X_d)$ has the same distribution as $X$ for any choice of signs
$\eta_1,\ldots,\eta_d\in \{-1,1\}$. 

Since this is only a matter of normalization we will also assume 
that $X$ is isotropic, which in this case means $\Ex X_i^2=1$ for $i=1,\ldots,d$.

By $\ve_i$ we will denote Bernoulli sequence, i.e. a sequence of i.i.d. symmetric $\pm 1$ r.v.'s,
we will also assume that variables $(\ve_i)_i$ are independent of $X$. By $(\cale_i)$ we
will denote a sequence of independent symmetric exponential r.v's with variance 1 (i.e. with
the density $\frac{1}{\sqrt{2}}\exp(-\sqrt{2}|x|)$).

Next lemma shows that vectors $(\ve_i)_{i\leq d}$ and $(\cale_i)_{i\leq d}$ are in a sense extremal in the class of 
$d$-dimensional unconditional isotropic log-concave vectors. 

\begin{lem}
Let $X$ be an isotropic unconditional log-concave vector.\\
i) For any $t\in \er^d$ and $p\geq 1$,
\begin{equation}
\label{compmomunc}
\frac{1}{\sqrt{2}}\Big\|\sum_{i=1}^d t_i\ve_i\Big\|_p\leq \Big\|\sum_{i=1}^d t_iX_i\Big\|_p
\leq 2\sqrt{6}\Big\|\sum_{i=1}^d t_i \cale_i\Big\|_p.
\end{equation}
ii) For any nonempty bounded set $T\subset \er^d$ we have
\begin{equation}
\label{compsuprad}
\Ex\sup_{t\in T}\sum_{i=1}^dt_i\ve_i\leq \sqrt{2}\Ex\sup_{t\in T}\sum_{i=1}^nt_iX_i.
\end{equation}
Moreover for any $\emptyset\neq I\subset\{1,\ldots,d\}$,
\begin{equation}
\label{compsupexp}
\Ex\sup_{t\in T}\sum_{i\in I}t_i\cale_i\leq C_2\log(|I|+1)\Ex\sup_{t\in T}\sum_{i\in I}t_iX_i.
\end{equation}
\end{lem}

\begin{proof}
i) By Lemma \ref{momgrow}, Jensen's inequality and unconditionality of $X$,
\[
\frac{1}{\sqrt{2}}\Big\|\sum_{i=1}^d t_i\ve_i\Big\|_p\leq
\Big\|\sum_{i=1}^d t_i\ve_i\Ex|X_i|\Big\|_p\leq \Big\|\sum_{i=1}^d t_i\ve_i|X_i|\Big\|_p
=\Big\|\sum_{i=1}^d t_iX_i\Big\|_p.
\]
On the other hand the result of Bobkov-Nazarov \cite{BN} and integration by parts give for $k=1,2,\ldots$,
$\|\sum_{i=1}^d t_iX_i\|_{2k}\leq \sqrt{6}\|\sum_{i=1}^d t_i\cale_i\|_{2k}$ and the upper bound in \eqref{compmomunc}
follows by Lemma \ref{momgrow}.

ii) Inequality \eqref{compsuprad} may be proven in a similar way as the lower bound in \eqref{compmomunc}. 
To finish the proof observe that
\begin{align*}
\Ex\sup_{t\in T}\sum_{i\in I}t_i\cale_i&=\Ex\sup_{t\in T}\sum_{i\in I}t_i\ve_i|\cale_i|
\leq \Ex \max_{i\in I}|\cale_i|\Ex\sup_{t\in T}\sum_{i\in I}t_i\ve_i\\
&\leq C\log(|I|+1)\Ex\sup_{t\in T}\sum_{i\in I}t_i\ve_i,
\end{align*}
thus \eqref{compsupexp} follows by \eqref{compsuprad}
\end{proof}

The next result easily follows by comparing  unconditional vectors with the exponential
random vector $\cale=(\cale_i)_{i\leq d}$.

\begin{prop}
\label{unclogd}
Suppose that $X$ is a $d$-dimensional log-concave unconditional vector. Then $X$ satisfies 
$\mathrm{SMP}(1/C\log(d+1))$.
\end{prop}

\begin{proof}
Recall that w.l.o.g.\ we assume that $X$ is isotropic. Let $p\geq 1$ and $T$ be a set in $\er^d$ with cardinality at least 
$e^p$ such that \eqref{sudass} holds. Then by \eqref{compmomunc} for distinct points $t,s\in T$, 
$\|\langle t-s,\cale\rangle\|_p\geq A/(2\sqrt{6})$, where $\cale=(\cale_i)_{i\leq d}$. We know (see Example \ref{indlogcon})
that $\cale$ satisfies the Sudakov minoration principle with a universal constant, thus by \eqref{compsupexp} we have
\[
\frac{A}{C}\leq \Ex\sup_{t,s\in T}\langle t-s,\cale\rangle \leq C_2\log(d+1)\Ex\sup_{t,s\in T}\langle t-s,X\rangle.
\] 
\end{proof}

We are now ready to present the main result of this section. Its proof is also based on comparison ideas, but in a less
straightforward way.

\begin{thm}
\label{uncpsquare}
Let $X$ be a log-concave unconditional vector in $\er^d$, $p\geq 1$ and $T\subset \er^d$ be such that $|T|\geq e^{p^2}$ and 
$\|\langle t-s,X\rangle\|_p\geq A$ for distinct points $t,s\in T$. Then 
\[
\Ex\sup_{t,s\in T}\langle t-s,X\rangle =2\Ex\sup_{t\in T}\langle t,X\rangle\geq \frac{1}{C}A.
\] 
\end{thm}

\begin{proof}
W.l.o.g.\ $X$ is isotropic. By Remark \ref{smallp} we may assume that $p\geq 2$. Observe also that if $0\in T$
then $\Ex\sup_{t,s\in T}\langle t-s,X\rangle\geq \Ex\sup_{t\in T}|\langle t,X\rangle|$, so for such $T$ it is enough to show
that
\begin{equation}
\label{supmod}
\Ex\sup_{t\in T}|\langle t,X\rangle| \geq \frac{1}{C}A.
\end{equation}

We divide the proof into 3 steps. In the first two steps we show that we may add
additional assumptions on the set $T$ (slightly decreasing its cardinality
and rescaling $A$ by a universal constant).

\medskip

\noindent
{\bf Step 1.}
We may assume that $0\in T$, $|T|\geq e^{p^2-p}$ and
\begin{equation}
\label{ass1}
\Big\|\sum_{i=1}^d t_{i}\ve_i\Big\|_p\leq \delta A \quad \mbox{ for  all } t\in T,
\end{equation}
where $\delta>0$ is a positive universal constant (to be chosen later).

Let $d_p(t,s)=\|\sum_{i=1}^d (t_i-s_i)\ve_i\|_p$.
By the result of Talagrand (see Example \ref{smprad}) we know that if $N(T,d_p,\alpha)\geq e^p$ then
$\Ex\sup_{t,s\in T}\sum_{i=1}^d (t_i-s_i)\ve_i\geq \frac{1}{C}\alpha$. Thus using \eqref{compsuprad}
we may assume that $N(T,d_p,\delta A/2)\leq e^p$, however this means that there exists
$t^{0}\in T$ such that 
\[
|\{t\in T\colon\ d_p(t,t^{0})\leq \delta A\}| \geq |T|/e^p.
\]
So we may consider the new set $T'=\{t-t^0\colon\ d_p(t,t^0)\leq \delta A\}.$ 

\medskip

\noindent
{\bf Step 2.} We may assume (changing $A$ into $A/2$)  that $0\in T$, $|T|\geq e^{p^2-p}$, \eqref{ass1} holds and 
\begin{equation}
\label{ass2}
|\supp (t)|\leq p \quad \mbox{ for all } t\in T. 
\end{equation}

By Step 1 we may assume
that $0\in T$ and \eqref{ass1} holds. The result of Hitczenko \cite{H} gives
\begin{equation}
\label{estrad}
\Big\|\sum_{i=1}^d t_{i}\ve_i\Big\|_p\geq 
\frac{1}{C_3}\Big(\sum_{i\leq p} t_i^*+\sqrt{p}\Big(\sum_{i> p} |t_i^*|^2\Big)^{1/2}\Big),
\end{equation}
where $t_i^*$ denotes the nonincreasing rearrangement of $(|t_i|)$.

Let us define $\varphi(x):=\sgn(x)(|x|-C_3\delta A/p)_{+}$ for $x\in \er$ 
and let $\varphi(t):=(\varphi(t_{i}))$ for $t\in \er^d$.
Then \eqref{ass1} and \eqref{estrad} imply  that $|\supp(\varphi(t))|\leq p$ for $t\in T$.
The upper bound in \eqref{compmomunc} and the Gluskin-Kwapie\'n estimate \cite{GK} yield
\begin{equation}
\label{BN}
\|\langle t, X\rangle\|_p\leq 
2\sqrt{6}\|\langle t, \cale\rangle\|_p\leq C_4(p\|t\|_{\infty}+\sqrt{p}\|t\|_2).
\end{equation}
Thus if $\delta\leq 1/(12C_3C_4)$ we get for $t\in T$,
\begin{align*}
\|\langle t-\varphi(t),X\rangle\|_p&\leq C_4(p\|t-\varphi(t)\|_{\infty}+\sqrt{p}\|t-\varphi(t)\|_{2})
\\
&\leq C_4\Big(2p\|t-\varphi(t)\|_{\infty}+\sqrt{p}\Big(\sum_{i>p}|t_i^*|^2\Big)^{1/2}\Big)
\\
&\leq C_3C_4\Big(2\delta A+\Big\|\sum_{i=1}^d t_{i}\ve_i\Big\|_p\Big)\leq 3C_3C_4\delta A\leq 
\frac{A}{4}.
\end{align*}
Therefore for any $t,s\in T$, $t\neq s$,
\[
\|\langle \varphi(t)-\varphi(s),X\rangle\|_p\geq \|\langle t-s,X\rangle\|_p -2\frac{A}{4}\geq 
\frac{A}{2}.
\]
Moreover the contraction principle for Rademacher processes (see Theorem 4.12 in \cite{LT}) and
the unconditionality of $X$ yield
\[
\Ex\sup_{t\in T}|\langle t, X\rangle| \geq \frac{1}{2}\Ex\sup_{t\in T}|\langle\varphi(t),X\rangle|,
\]
so it is enough to show estimate \eqref{supmod} for the set $\varphi(T)=(\varphi(t))_{t\in T}$.
Note that condition \eqref{ass1} holds for $\varphi(T)$ since it holds for $T$ and $|\varphi(t_i)|\leq |t_i|$ for
all $i$.
\medskip

\noindent
{\bf Step 3.} We consider a finite set $T$ such that $0\in T$, $|T|\geq e^{p^2-p}\geq e^{p^2/2}$,
$\|\langle t-s,X\rangle\|_p\geq A$ for distinct points $t,s\in T$ and conditions \eqref{ass1}-\eqref{ass2} hold. 
To finish the proof it is enough to show \eqref{supmod}.

To this end we construct
inductively points $t_1,\ldots,t_{N}$. For $t_1$ we take any point in $T$. Suppose that
$t_1,\ldots,t_n$ are constructed. We put 
\[
I_n:=\bigcup_{k\leq n}\supp(t_k)\quad \mbox{ and }\quad J_n:=\{1,\ldots,d\}\setminus I_n
\]
and  consider the set
\[
T_n:=\Big\{t\in T\colon \big\|\langle t_{J_n},X\rangle \big\|_p\geq \frac{A}{4}\Big\},
\]  
where for $I\subset \{1,\ldots,d\}$ we put $t_I:=(t_i\ind_{\{i\in I\}})$.
If $T_n$ is nonempty we pick for $t_{n+1}$ any point in this set, otherwise we finish the construction
and set $n=N$, $I=I_N$, $J=J_N$. 

We distinguish between two possibililities.

\medskip
\noindent
\textbf{Case I.} $N\leq e^p$, then $|I|\leq \sum_{i=1}^N|\supp(t_k)|\leq Np\leq e^{2p}$.
Observe that then for any $t,s\in T$, $t\neq s$,
\[
\|\langle t_I-s_I,X\rangle\|_p\geq \|\langle t-s,X\rangle\|_p-\|\langle t_J,X\rangle\|_p-\|\langle s_J,X\rangle\|_p
\geq \frac{A}{2}.
\]
Thus by \eqref{BN},\eqref{estrad} and \eqref{ass1},
\begin{align*}
\frac{A}{2}&\leq \|\langle t_I-s_I,X\rangle\|_p\leq C_4(p\|t_I-s_I\|_\infty+\sqrt{p}\|t_I-s_I\|_2)
\\
&\leq C_4\Big(2p\|t_I-s_I\|_\infty+C_3\Big\|\sum_{i\in I}t_i\ve_i\Big\|_p\Big)\leq 
C_3C_4(2p\|t_I-s_I\|_{\infty}+\delta A)
\end{align*}
therefore if $\delta\leq 1/(4C_3C_4)$ then $\|t_I-s_I\|_\infty\geq A/(8C_3C_4p)$.
By the Gluskin-Kwapie\'n estimate \cite{GK} we get
\[
\|\langle t_I-s_I,\cale\rangle\|_{p^2/2}\geq \frac{p^2}{C}\|t_I-s_I\|_\infty\geq \frac{pA}{C}.
\] 
where $\cale=(\cale_i)_{i\leq d}$. Since $\cale$ satisfies the Sudakov minoration principle with a uniform constant 
(see Example \ref{indlogcon}) and $|T|\geq e^{p^2/2}$ we obtain
\[
2\Ex\sup_{t\in T}|\langle t_I,\cale\rangle|\geq \Ex\sup_{t\in T}\langle t_I-s_I,\cale\rangle\geq \frac{pA}{C}.
\]
Thus by \eqref{compsupexp} we have
\[
\frac{pA}{C}\leq \Ex\sup_{t\in T}|\langle t_I,\cale\rangle| \leq C_2\log(|I|+1)\Ex\sup_{t\in T}|\langle t_I,X\rangle|
\leq Cp\Ex\sup_{t\in T}|\langle t,X\rangle|.
\] 

\medskip
\noindent
\textbf{Case II.} $N\geq e^p$. Let $I_0=\emptyset$, $\Delta_k:=I_k\setminus I_{k-1}$ and
 $s_k:=t_{k,\Delta_k}$ for $k=1,\ldots,N$. Then by our construction vectors $s_k$ have
disjoint supports and $\|\langle s_k,X\rangle\|_p\geq A/4$ for $k=1,\ldots,N$. Thus by  Proposition \ref{supcoord}
(applied with $V=A/4$, $a_i=|s_i|$ and isotropic vector $(\langle s_i,X\rangle/|s_i|)_{i\leq N}$) we get
\[
\Ex\max_{k\leq N}|\langle s_k,X\rangle|\geq \frac{A}{4C_1}. 
\]
Since sets $\Delta_k$ are disjoint and vector $X$ is unconditional we get
\[
\Ex\max_{t\in T}|\langle t,X\rangle|\geq \frac{1}{2}\Ex\max_{t\in T}\max_{k}|\langle t_{\Delta_k},X\rangle|
\geq \frac{1}{2}\Ex\max_{k\leq N}|\langle s_k,X\rangle|\geq \frac{A}{8C_1}.
\]
\end{proof}

\begin{rem}
Following Remark \ref{equivcov} we may restate Theorem \ref{uncpsquare} in terms of covering numbers -- for any log-concave
unconditional vector $X$, any nonempty $T\subset \er^d$ and $A>0$,
\[
\Ex\sup_{t,s\in T}\langle t-s,X\rangle =2\Ex\sup_{t\in T}\langle t,X\rangle\geq 
\frac{1}{C}\sup_{p\geq 1}\min\Big\{\frac{A}{p}\sqrt{\log N(T,d_{X,p},A)},A\Big\}.
\]
\end{rem}

\section{Invariant log-concave vectors}

In this section we investigate the class of invariant log-concave vectors. First result shows
that $p$-th moments of norms of such vectors are almost constant for $p\leq d$.

\begin{prop}
\label{compmominv}
Let $K$ be a symmetric convex body in $\er^d$ and $X$ be a random $d$-dimensional vector with the density of the form
$e^{-\varphi(\|x\|_K)}$, where $\varphi\colon [0,\infty)\rightarrow (-\infty,\infty]$ is a nondecreasing convex function.
Then
\[
(\Ex\|X\|_K^d)^{1/d}\leq C_5\Med(\|X\|_K).
\]
\end{prop}

\begin{proof}
Let $\mu$ denotes the law of $X$ and $m:=\Med(\|X\|_K)$. Then
\[
\frac{1}{2}=\mu(mK)=\int_{mK}e^{-\varphi(\|x\|_K)}\de x\leq \int_{mK}e^{-\varphi(0)}\de x=e^{-\varphi(0)}m^d\vol(K)
\]
and
\begin{align*}
1&\geq \mu(2emK)=\int_{2emK}e^{-\varphi(\|x\|_K)}\de x\geq \int_{2emK}e^{-\varphi(2em)}\de x
\\
&=e^{-\varphi(2em)}(2em)^d\vol(K).
\end{align*}
Therefore 
\[
e^{\varphi(2em)-\varphi(0)}\geq \frac{1}{2}(2e)^d\geq e^d.
\]
Convexity of $\varphi$ implies that
\begin{equation}
\label{estvarphi}
\varphi(r)-\varphi(m)\geq \frac{d}{2em}(r-m) \quad \mbox{ for }r\geq 2em.
\end{equation}

Integrating in polar-type coordinates we get
\[
\frac{1}{2}=\mu(mK)=c_K\int_0^m e^{-\varphi(r)}r^{d-1}\de r\geq c_K\int_0^m e^{-\varphi(m)}r^{d-1}\de r
=\frac{c_K}{d}m^de^{-\varphi(m)},
\]
where $c_K:=d\,\vol(K)$.
Hence for $s\geq 0$,
\begin{align*}
\Pr(\|X\|_K\geq sm)
&=\mu(\er^d\setminus smK)=c_K\int_{sm}^{\infty}e^{-\varphi(r)}r^{d-1}\de r
\\
&\leq
\frac{d}{2}m^{-d}\int_{sm}^{\infty}e^{\varphi(m)-\varphi(r)}r^{d-1}\de r.
\end{align*}
Using \eqref{estvarphi} we get 
\[
\Pr(\|X\|_K\geq sm)\leq \frac{d}{2}m^{-d}e^{\frac{d}{2e}}\int_{sm}^{\infty}e^{-\frac{dr}{2em}}r^{d-1}\de r
\quad \mbox{ for }s\geq 2e.
\]

The function $r\mapsto e^{-\frac{dr}{4em}}r^{d-1}$ is decreasing for $r\geq 4\frac{d-1}{d}em$, thus
\[
e^{-\frac{dr}{2em}}r^{d-1}\leq e^{-d}(4em)^{d-1}e^{-\frac{dr}{4em}} \quad \mbox{ for }r\geq 4em.
\]
Therefore for $s\geq 4e$,
\[
\Pr(\|X\|_K\geq sm)\leq \frac{d}{2}(em)^{-d}(4em)^{d-1}e^{\frac{d}{2e}}\int_{sm}^{\infty}e^{-\frac{dr}{4em}}\de r
=\frac{1}{2}e^{\frac{d}{2e}}4^de^{-\frac{sd}{4e}}.
\]
Integrating by parts we get
\begin{align*}
\Ex\|X\|_K^d
&\leq (4em)^d+dm^d\int_{4e}^\infty s^{d-1}\Pr(\|X\|_K\geq sm)\de s
\\
&\leq
(4em)^d+\frac{d}{2}e^{\frac{d}{2e}}(4m)^d \int_{0}^\infty s^{d-1}e^{-\frac{sd}{4e}}\de s
\end{align*}
and it easily follows that $(\Ex\|X\|_K^d)^{1/d}\leq C_5m$.
\end{proof}

It turns out that the Sudakov minoration property holds with almost the same constant for all
vectors in the same class of invariant log-concave vectors.

\begin{thm}
\label{inv1}
Let $X^i$, $i=1,2$, be $d$-dimensional random vectors with densities of the form
$e^{-\varphi_i(\|x\|_K)}$, where $K$ is a symmetric convex body in $\er^d$ and
$\varphi_i\colon [0,\infty)\rightarrow (-\infty,\infty]$ are nondecreasing convex functions. 
If $X^1$ satisfies  $\mathrm{SMP}(\kappa)$ then $X^2$ satisfies $\mathrm{SMP}(\kappa/C_6)$.
\end{thm}

\begin{proof}
Observe that $X^i$ has the same distribution as $R_iY$, where $Y$ is uniformly distributed on $K$ and $R_i$ are
nonnegative r.v's independent of $Y$. 
We have $\|X^i\|_K\leq R_i$, in particular $\Ex R_i\geq \frac{1}{2}\Med(R_i)\geq \frac{1}{2}\Med(\|X^i\|_K)$.
Moreover 
\[
(\Ex \|X^i\|_K^d)^{1/d}=(\Ex R_i^d)^{1/d}(\Ex\|Y\|_K^d)^{1/d}\geq \frac{1}{2}(\Ex R_i^{d})^{1/d}.
\]
Therefore Proposition \ref{compmominv} implies 
\[
\Ex R_i\leq \|R_i\|_p\leq \|R_i\|_d\leq 4C_5\Ex R_i \quad \mbox{ for }1\leq p\leq d.
\]

We need to show that $X^2$ satisfies $\mathrm{SMP}_p(\kappa/C)$ for $p\geq 1$. By Lemma  \ref{red_belowd} it
is enough to consider only $p\leq d$. Since it is a matter of scaling we may assume that $\Ex R_i=1$ for $i=1,2$.
For any $T\subset \er^d$ we then have
\[
\Ex\sup_{t,s\in T}\langle t-s,X^i\rangle=\Ex R_i\Ex\sup_{t,s\in T}\langle t-s,Y\rangle
=\Ex\sup_{t,s\in T}\langle t-s,Y\rangle.
\]
Moreover for $p\geq 1$,
\[
\|\langle u,X^i\rangle\|_p=\|R_i\|_p\|\langle u,Y\rangle\|_p \quad \mbox{ for any }u\in \er^n.
\]

Fix $1\leq p\leq d$ and take $T\subset \er^d$ with $|T|\geq e^p$ such that  
$\|\langle t-s,X^2\rangle\|_p\geq A$ for all $t,s\in T$, $t\neq s$. Then 
$\|\langle t-s,X^1\rangle\|_p\geq \frac{1}{4C_5}A$ for distinct points $t,s\in T$ and $\mathrm{SMP}$ for $X^1$ yields
\[
\Ex\sup_{t,s\in T}\langle t-s,X^2\rangle=\Ex\sup_{t\in T}\langle t-s,X^1\rangle\geq \frac{\kappa}{4C_5}A.
\]
\end{proof}

As a corollary we  show that a large class of invariant log-concave vectors satisfy SMP with a
universal constant.

\begin{cor}
\label{thm_inv2}
All $d$-dimensional random vectors with densities of the form $\exp(-\varphi(\|x\|_p))$, where
$1\leq p\leq \infty$ and  $\varphi\colon [0,\infty)\rightarrow (-\infty,\infty]$ is nondecreasing and convex
satisfy Sudakov minoration principle with a universal constant. In particular all rotationally invariant log-concave
random vectors satisfy Sudakov minoration principle with a universal constant.
\end{cor}

\begin{proof}
We apply Theorem \ref{inv1} with $X^2=X$ and $X^1$ having the density of the form $c_p^d\exp(-\varphi_p(\|x\|_p))$,
where $\varphi_p(r)=r^p$ for $p<\infty$ and $\varphi_{\infty}=\infty \ind_{[1,\infty)}$. Note that $X^1$ is log-concave with
a product density so it satisfies SMP with a universal constant (see Example \ref{indlogcon}).
\end{proof}

\section{Applications}

In the last section we apply chaining techniques to show properties of vectors satisfying SMP. 
Before we formulate our results we state a simple general estimate for moments of suprema of stochastic processes
based on chaining.

\begin{prop}
\label{chaining}
Let $(X_t)_{t\in T}$ be a stochastic process and $(T_k)_{0\leq k\leq k_1}$ be a sequence of subsets of $T$
such that $|T_k|\leq e^{2^{k+1}}$ for 
$0\leq k\leq k_1$ and $T_{k_1}=T$. Moreover suppose that $\pi_k\colon\ T\rightarrow T_k$
for $0\leq k\leq k_1$ and $\pi_{k_1}(t)=t$ for all $t\in T$. Then for any $1\leq k_0\leq k_1-1$ and 
$2^{k_0-1}\leq p\leq 2^{k_0}$, 
\[
\Big\|\sup_{t\in T_k}|X_t|\Big\|_p\leq 
3e^3\Big(\sup_{t\in T}\sum_{k=k_0+1}^{k_1}\|X_{\pi_k(t)}-X_{\pi_{k-1}(t)}\|_{2^{k}}+\sup_{t\in T_{k_0}}\|X_t\|_p\Big).
\]
\end{prop}

\begin{proof}
Define 
\[
m(l):=\sup_{t\in T}\sum_{k=l+1}^{k_1}\|X_{\pi_k(t)}-X_{\pi_{k-1}(t)}\|_{2^{k}}.
\]
Then for $u\geq 2e^{3}$,
\begin{align}
\notag
\Pr\Big(&\sup_{t\in T}|X_t-X_{\pi_{k_0}(t)}|\geq um(k_0)\Big)
\\
\notag
&\leq \Pr\Big(\sup_{t\in T}\sum_{k=k_0+1}^{k_1}|X_{\pi_{k}(t)}-X_{\pi_{k-1}(t)}|\geq um(k_0)\Big)
\\
\notag
&\leq \Pr\Big(\exists_{k_0+1\leq k\leq k_1}\exists_{t\in T}\
|X_{\pi_{k}(t)}-X_{\pi_{k-1}(t)}|\geq u\|X_{\pi_{k}(t)}-X_{\pi_{k-1}(t)}\|_{2^k}\Big)
\\
\notag
&\leq \sum_{k=k_0+1}^{k_1}\sum_{s\in T_{k}}\sum_{s'\in T_{k-1}}
\Pr(|X_{s}-X_{s'}|\geq u\|X_s-X_{s'}\|_{2^k})
\\
\label{smartunion}
&\leq \sum_{k=k_0+1}^{k_1}|T_{k}||T_{k-1}| u^{-2^k}
\leq \sum_{k=k_0+1}^{k_1}\Big(\frac{e^3}{u}\Big)^{2^k}
\leq  2\Big(\frac{e^3}{u}\Big)^{2^{k_0+1}}\leq 2\Big(\frac{e^3}{u}\Big)^{2p}.
\end{align}
Hence integrating by parts we get
\begin{align*}
\Ex\sup_{t\in T}|X_t-X_{\pi_{k_0}(t)}|^p
&\leq (e^{3}m(k_0))^p\Big(2^p+p\int_2^{\infty}u^{p-1}2u^{-2p}\de u\Big)
\\
&=(2e^{3})^p(1+2^{1-2p})m(k_0)^p\leq (3e^3m(k_0))^p.
\end{align*}
Moreover
\[
\Ex\sup_{t\in T}|X_{\pi_{k_0}(t)}|^p\leq \sum_{t\in T_{k_0}}\Ex|X_t|^p\leq |T_{k_0}|\sup_{t\in T_{k_0}}\|X_t\|_p^p
\leq e^{4p}\sup_{t\in T_{k_0}}\|X_t\|_p^p.
\]
Hence
\begin{align*}
\Big\|\sup_{t\in T}|X_t|\Big\|_p&\leq \Big\|\sup_{t\in T}|X_t-X_{\pi_{k_0}(t)}|\Big\|_p
+\Big\|\sup_{t\in T}|X_{\pi_{k_0}(t)}|\Big\|_p
\\
&\leq 3e^3\Big(m(k_0)+\sup_{t\in T_{k_0}}\|X_t\|_p\Big).
\end{align*}
\end{proof}

\begin{rem}
\label{remchain}
If $2^{k_0-1}\leq p\leq 2^{k_0}$, but $k_0\geq k_1$ then $p\geq 2^{k_1-1}$ and
\[
\Big\|\sup_{t\in T}|X_t|\Big\|_p\leq |T|^{1/p}\sup_{t\in T}\|X_t\|_p\leq e^4\sup_{t\in T}\|X_t\|_p.
\]
\end{rem}

Now we show that if weak moments of a random vector $X$ with SMP property dominate weak moments of another 
random vector $Y$ then strong moments of $X$ dominate strong moments of $Y$ up to a logarithmic factor.

\begin{prop}
\label{weakstrongcomp}
Suppose that a random vector $X$ in $\er^d$ satisfies $\mathrm{SMP}(\kappa)$. Let $Y$ be a random $d$-dimensional vector
such that $\|\langle t,Y\rangle\|_p\leq \|\langle t,X\rangle\|_p$ for all $p\geq 1$, $t\in \er^d$. 
Then for any norm $\|\ \|$ on $\er^d$ and $p\geq 1$, 
\begin{align}
\notag
(\Ex\|Y\|^p)^{1/p}
&\leq C\Big(\frac{1}{\kappa}\log_{+}\Big(\frac{ed}{p}\Big)\Ex\|X\|
+\sup_{\|t\|_*\leq 1}(\Ex|\langle t,Y\rangle|^p)^{1/p}\Big)
\\
\label{weakstr}
&\leq C\Big(\frac{1}{\kappa}\log_+\Big(\frac{ed}{p}\Big)+1\Big)(\Ex\|X\|^p)^{1/p}.
\end{align}
\end{prop}

\begin{proof}
Let $T$ be a $1/2$-net in $B_{\|\ \|_*}:=\{t\in \er^d\colon\ \|t\|_*\leq 1\}$ of cardinality at most $5^d$, then
$\|x\|\leq 2\max_{t\in T}\langle t,x\rangle$ for any $x\in \er^d$. For $q\geq 1$ choose a maximal
set $S_q\subset T$ such that $\|\langle t-s,X\rangle\|_q> \frac{2}{\kappa}\Ex\|X\|$ for all distinct points $t,s\in S_q$. 
Since $X$ satisfies $\mathrm{SMP}(\kappa)$ and
\[
\Ex\max_{t,s\in S_q}\langle t-s,X\rangle\leq \Ex\sup_{\|u\|_*\leq 2}\langle u,X\rangle=2\Ex\|X\|
\]
we get that $|S_q|< e^q$. 

Let $k_1$ be the smallest integer such that $2^{k_1+1}\geq d\log 5$. Put $T_{k_1}:=T$ and $T_k:=S_{2^{k+1}}$ for
$0\leq k\leq k_1$. Let $\tilde{\pi}_k\colon T\rightarrow S_{2^{k+1}}$ be such that for any $t\in T$,
$\|\langle t-\tilde{\pi}_k(t),X\rangle\|_{2^{k+1}}\leq \frac{2}{\kappa}\Ex\|X\|$.
Define maps $\pi_k\colon T\rightarrow T_{k}$ by
$\pi_{k_1}(t)=t$ and
for $0\leq k< k_1$ by $\pi_k:=\tilde{\pi}_k\circ\tilde{\pi}_{k+1}\circ\ldots\circ\tilde{\pi}_{k_1-1}$.
Let $k_0$ be the smallest positive integer such that $2^{k_0}\geq p$.
 
If  $k_0\leq k_1-1$ we may apply
Proposition \ref{chaining} to $X_t:=\langle t,Y\rangle$. Hence
\begin{align*}
(\Ex\|Y\|^p)^{1/p}&\leq 2\Big\|\max_{t\in T}|\langle t,Y\rangle|\Big\|_p
\\
&\leq 6e^3\Big(\sup_{t\in T}\sum_{k=k_0+1}^{k_1}\|\langle \pi_k(t)-\pi_{k-1}(t),Y\rangle\|_{2^{k}}+
\sup_{t\in T_{k_0}}\|\langle t,Y\rangle \|_p\Big).
\end{align*}
To show the first inequality in \eqref{weakstr} it is enough to notice that   
$k_1-k_0\leq C\log(ed/p)$ and for any $1\leq k\leq k_1$ we have
\begin{align*}
\sup_{t\in T}\|\langle \pi_k(t)-\pi_{k-1}(t),Y\rangle\|_{2^{k}}
&\leq \sup_{t\in T}\|\langle \pi_k(t)-\pi_{k-1}(t),X\rangle\|_{2^{k}}
\\
&\leq \sup_{t\in T}\|\langle t-\tilde{\pi}_{k-1}(t),X\rangle\|_{2^{k}}\leq \frac{2}{\kappa}\Ex\|X\|.
\end{align*}

If $k_0\geq k_1$ we use Remark \ref{remchain} instead of Proposition \ref{chaining}. 

The second inequality in \eqref{weakstr} follows since
\[
\sup_{t\in T_{k_0}}\|\langle t,Y\rangle \|_p\leq \sup_{\|t\|_*\leq 1}\|\langle t,X\rangle \|_p\leq (\Ex\|X\|^p)^{1/p}.
\]
\end{proof}

The next statement states that weak and strong moments of random vectors with SMP property are comparable up to 
a logarithmical factor.

\begin{cor}
Suppose that $X$ is a $d$-dimensional random vector, which satisfies $\mathrm{SMP}(\kappa)$. Then for any norm
$\|\ \|$ on $\er^d$ and any $p\geq 1$,
\[
\big(\Ex\|X\|^p\big)^{1/p}\leq 
C\Big(\frac{1}{\kappa}\log_+\Big(\frac{ed}{p}\Big)\Ex\|X\|+\sup_{\|t\|_*\leq 1}\big(\Ex|\langle t,X\rangle|^p\big)^{1/p}\Big).
\] 
\end{cor}

\begin{proof}
We apply Proposition \ref{weakstrongcomp} with $Y=X$.
\end{proof}

The last result shows that for a class of invariant vectors we may eliminate logarithmic factors.

\begin{prop}
\label{weakstronginv}
Let $X$ be a $d$-dimensional random vector with the density of the form
$e^{-\varphi(\|x\|_r)}$, where $1\leq r<\infty$ and 
$\varphi\colon [0,\infty)\rightarrow (-\infty,\infty]$ is nondecreasing and convex. 
Let $Y$ be a random vector
in $\er^d$ such that $\|\langle t,Y\rangle\|_p\leq \|\langle t,X\rangle\|_p$ for $p\geq 1$. Then
for any norm $\|\ \|$ on $\er^d$ and $p\geq 1$, 
\[
(\Ex\|Y\|^p)^{1/p}\leq C(r)\Ex\|X\|+C\sup_{\|t\|_*\leq 1}\big(\Ex|\langle t,Y\rangle|^p\big)^{1/p}
\leq C'(r)(\Ex\|X\|^p)^{1/p},
\]
where $C(r)$ and $C'(r)$ depend only on $r$.
In particular for any norm $\|\ \|$ on $\er^d$ and any $p\geq 1$,
\[
\big(\Ex\|X\|^p\big)^{1/p}\leq 
C(r)\Ex\|X\|+C\sup_{\|t\|_*\leq 1}\big(\Ex|\langle t,X\rangle|^p\big)^{1/p}.
\] 
\end{prop} 

\begin{proof}
Let $\tilde{X}$ have the density of the form $c_p^d\exp(-\|x\|_p^p)$. Then $\tilde{X}$ has independent
coordinates. We have $X=RZ$ and $\tilde{X}=\tilde{R}Z$, where $Z$ is uniformly distributed on $B_p^d$ and
$R,\tilde{R}$ are nonnegative random variables independent of $Z$. Since it is only a matter of normalization
we may assume that $\Ex R=\Ex\tilde{R}$. Then $\Ex\|X\|=\Ex R\Ex\|Z\|=\Ex\tilde{R}\Ex\|Z\|=\Ex\|\tilde{X}\|$.
Moreover, Proposition \ref{compmominv} easily implies (see the proof of Theorem \ref{inv1}) that for
$1\leq p\leq d$, $\|R\|_p\leq \|R\|_d\leq 4C_5\|R\|_1=4C_5\|\tilde{R}\|_1\leq 4C_5\|\tilde{R}\|_p$. Thus for
any $t\in \er^d$ and $1\leq p\leq d$,
\[
\|\langle t,X\rangle\|_p=\|R\|_p\|\langle t,Z\rangle\|_p\leq 4C_5\|\tilde{R}\|_p\|\langle t,Z\rangle\|_p
=4C_5\|\langle t,\tilde{X}\rangle\|_p.
\]
For $t\in \er^d$ and $d\leq p\leq d\log 5$ we get by Lemma \ref{momgrow}
\[
\|\langle t,X\rangle\|_p\leq \log 5\|\langle t,X\rangle\|_d\leq 2C_5\log 5\|\langle t,\tilde{X}\rangle\|_d
\leq 2C_5\log 5\|\langle t,\tilde{X}\rangle\|_p.
\]
Therefore we have
\begin{equation}
\label{compmomY}
\|\langle t,Y\rangle\|_p\leq \|\langle t,X\rangle\|_p\leq 2C_5\log 5\|\langle t,\tilde{X}\rangle\|_p
\quad\mbox{for }1\leq p\leq d\log 5,\ t\in\er^d.
\end{equation}

Let $T$ be a $1/2$-net in $B_{\|\ \|_*}$ of cardinality at most $5^d$. Let $k_1$ be the smallest integer such that 
$e^{2^{k_1+1}}\geq 5^d$. 
By the result of Talagrand \cite{Ta2} we may find sets
$T_k\subset T$, $0\leq k\leq k_1$ and maps $\pi_k\colon T\mapsto T_k$ such that $T_{k_1}=T$, 
$|T_k|\leq e^{2^{k+1}}$, $\pi_{k_1}(t)=t$ for $t\in T$ and
\[
\sum_{k=1}^{k_1} \|\langle \pi_{k}(t)-\pi_{k-1}(t),\tilde{X}\rangle\|_{2^k}\leq 
\frac{1}{2}C(r)\Ex\sup_{t\in T}\langle t,\tilde{X}\rangle
\leq C(r)\Ex\|\tilde{X}\|=C(r)\Ex\|X\|.
\]

We may now proceed as in the proof of Proposition \ref{weakstrongcomp} observing that by \eqref{compmomY} we have
$\|\langle \pi_{k}(t)-\pi_{k-1}(t),Y\rangle\|_{2^k}\leq 2C_5\log 5\|\langle \pi_{k}(t)-\pi_{k-1}(t),\tilde{X}\rangle\|_{2^k}$
for $0\leq k\leq k_1$.
\end{proof}

\begin{rem}
Using the two-sided bound for the expected value of suprema of Bernoulli processes \cite{BL} 
one may show that Proposition \ref{weakstronginv} is also satisfied in the case $r=\infty$.
\end{rem}

\noindent
Institute of Mathematics\\
University of Warsaw\\
Banacha 2\\
02-097 Warszawa\\
Poland\\
\texttt{rlatala@mimuw.edu.pl}

\end{document}